\documentclass[12pt,reqno]{amsart}
\usepackage{amsmath}
 \usepackage{amssymb,latexsym, esint}
 \usepackage{amsthm}
 \usepackage[latin1]{inputenc}
 \usepackage{enumerate}
\usepackage{color}

\newtheorem{thm}{Theorem}[section]
\newtheorem{theorem}{Theorem}[section]
 \newtheorem{prop}{Proposition}[section]

\newtheorem{lemma}[theorem]{Lemma}

\theoremstyle{definition}
\newtheorem{defi}{Definition}[section]

\theoremstyle{remark}

\newtheorem{rems}{Remarks}[section]

\begin{document}
\title[Bounds on the fundamental gap]{Optimal bounds on the 
fundamental spectral gap with single-well potentials}

\author{Evans M. Harrell II}
\address{School of Mathematics,
Georgia Institute of Technology,\
Atlanta GA 30332-0160, USA.} 
\email{ harrell@math.gatech.edu}

\author{Zakaria El Allali}
\address{Team of Modeling and Scientific Computing, 
Department of Mathematics and Computer,
Faculty Multidisciplinary of Nador, University Mohammed Premier, Morocco.}
\email{z.elallali@ump.ma}

\date{July 21, 2018}


\begin{abstract}
We characterize the potential-energy functions $V(x)$ that 
minimize the gap $\Gamma$
between the two lowest
Sturm-Liouville eigenvalues for
\[
H(p,V) u :=  -\frac{d}{dx}
\left(p(x)\frac{du}{dx}\right)+V(x) u  = \lambda u, \quad\quad x\in [0,\pi ], 
\]
where separated self-adjoint
boundary conditions are imposed at end points, and $V$
is subject to various assumptions, especially convexity or having a ``single-well'' form.
In the classic case where $p=1$ we recover with different arguments the result of Lavine that
$\Gamma$ is uniquely minimized among convex $V$ by the constant, and in the case of
single-well potentials, with no restrictions on the position of the minimum,
we obtain a new, sharp bound, that  $\Gamma > 2.04575\dots$.
\end{abstract}

\maketitle

\section{Introduction}

In this article, we consider a
Sturm-Liouville  operator on a finite interval, scaled without loss of generality to have length $\pi$,
\begin{equation}\label{eP}
H(p,V) u :=  -
\frac{d}{dx}\left(p(x)\frac{du}{dx}\right)+V(x) u  = 
\lambda u, \quad\quad x\in [0,\pi], 
\end{equation}
where $p(x)$ is a bounded $C^1$ function that is uniformly positive 
on $(0, \pi)$
and various assumptions are made on the potential energy $V(x)$, for instance that
$V(x) = V_0 \pm V_1$, where $V_1$ is either convex or
of ``single-well'' form and $V_0$ is fixed.
We shall always make assumptions so that $H$ is self-adjoint with purely discrete spectrum.  According to 
\cite{Wei}, \S 8.4, this is guaranteed without further conditions when $V(x) \ge C > -\infty$ is in the limit-point case, while if 
$V(x) \ge C > -\infty$ is in the limit-circle case, we may
impose any separated homogeneous boundary conditions of the form
\begin{align}\label{bc}
u(0) \cos \alpha - (pu')(0) \sin \alpha &= 0;\nonumber\\
u(\pi) \cos \beta - (pu')(\pi) \sin \beta &= 0,
\end{align}
where $0\leq \alpha , \beta < \pi $ (In particular this encompasses the possibility of either Dirichlet or Neumann boundary conditions. 
See \cite{Zet}, \S 4.6 for further discussion of the boundary conditions for Sturm-Liouville problems.)
If $V(x)$ is not bounded from below, further conditions are required.  A sufficient but far from necessary 
condition that is good for our purposes is that 
\begin{equation}
V(x) \ge C - \frac{1}{4 x} - \frac{1}{4 (\pi-x)}
\end{equation}
for some $C > - \infty$.
Again, in the limit-circle case, we may fix any of the boundary conditions 
\eqref{bc} to make $H$ self-adjoint.

We denote the
eigenvalues  $\{\lambda_n\}$, $n = 1, 2, \dots$,
and arrange them in
nondecreasing order:
$$\lambda_1 \leq \lambda_2 \leq \lambda_3 \ldots \lambda_m \leq \cdots$$
The function $V$ is called a
{\em single-well function} if $V$ is non-increasing  on $[0,a ] $ and 
non-decreasing on $[a,\pi ] $
for some $a \in [0, \pi]$. The point $a$ is called a
{\em transition point}.  We do not assume that $a$ is unique.

The functional $\Gamma(V)$ defined as
$$\Gamma(V)=\lambda_2(V)-\lambda_1(V)$$
is called the
{\em fundamental gap}, 
of interest in quantum mechanics 
as the excitation energy to raise a particle out of its ground state.
There are both physical and mathematical 
reasons to study $\Gamma(V)$.  Most frequently in the literature, $p \equiv 1$,
that is, the equation is in Liouville normal form, sometimes
referred to as Schr\"odinger form.  (Given that any equation
of the form \eqref{eP} can be put into this form
(cf. \cite{BiRo}, \S10.9 or \cite{Tes}, \S 9.1), we observe that this transformation alters the spectral problem,
because the required transformation
depends on the eigenparameter.  Hence the inclusion of the 
function $p$ in the leading term is a true generalization.)
Although $\lambda_1$ is always simple, so that 
$\Gamma > 0$, there is no positive lower bound on $\Gamma$ without assumptions on $V$, as exemplified by 
double-well potentials such as $V = M \chi_{[\frac{\pi}{3},\frac{2 \pi}{2}]}$ where $M$ is arbitrarily large, e.g., \cite{AsHa82}.
In \cite{AsBe} Ashbaugh and Benguria found that the optimal lower bound for $\Gamma$ for
symmetric single-well potentials and Dirichlet boundary conditions
is achieved if and only if $V$ is constant on $(0,\pi)$.
In \cite{AsHaSv}
Ashbaugh, Harrell and  Svirsky studied the optimization 
of $\Gamma(V)$ under $L^p$ constraints (even in $n$ dimensions) by first proving existence of optimizers and then 
carrying out a variational analysis, a strategy that we shall also employ here.  An advance was made by 
Lavine in 1994 \cite{Lav}, who considered the class of convex potentials on $[0, \pi]$ and proved that
with either Dirichlet or Neumann boundary conditions, $\Gamma$ attains its minimum if and only of $V$
is constant.  Later, in 2002, Horv\'ath \cite{Horvath} returned with Lavine's methods to the problem of single-well potentials,
without symmetry assumptions, but assuming a transition point at $\frac{\pi}{2}$, and showed that
the gap for the Dirichlet problem is minimized when the potential is constant.   In 2015 Yu and Yang \cite{YuYa}
extended Horv\'ath's result by allowing other transition points (under a technical condition) and both Dirichlet and Neumann conditions.
In this article, we provide lower bounds for the gap between the first two eigenvalues of the problems $(\ref{eP})$  with general
single-well potential $V(x)$ with a transition point $a \in [0, \pi]$, without any restriction on $a$, and also for the case where the potential is convex.  We are furthermore able to analyze the case where $V = V_0 + V_1$, where $V_0$ is fixed 
and $V_1$ is assumed either single-well or convex.  In contrast to the earlier studies 
of single-well potentials, which restrict the transition point in one way or other,
the minimizing potentials we find
are {\em not} in general constant, although if extra conditions are imposed locating the 
transition point sufficiently far from $0$ or $\pi$, 
our methods can lead to constant potentials.  Although we do not pursue this idea here in detail,
this remark is a way of 
understanding the results in \cite{AsBe,Horvath,YuYa}.

\section{Classes of Potentials}

We define here several classes of functions that will play a role below.  

\begin{defi}\label{fnclasses}
For $0 \le M \le \infty$, let 

$\mathcal{S}_M := \{f(x): 0 \le f(x) \le M$: there exists $a \in [0, \pi]$
such that $f(x)$ is nonincreasing
for $x \le a$ and nondecreasing for
$x \ge a\}$.  We identify functions in $\mathcal{S}_M$ when they are equal a.e.

$\mathcal{C}_M := \{f(x): 0 \le f(x) \le M$: $f(x)$ is convex on $[0, \pi]$\}.
\end{defi}

Each of these sets has a useful compactness property when $M$ is bounded:

\begin{prop}\label{subseq}
Let $\mathcal{A}$ denote any of the sets in 
Definition {\rm\ref{fnclasses}} with $M < \infty$.
For any sequence $f_n \in \mathcal{A}$ there exist
a subsequence
$f_{n_k}$ and a function $f_\star$ such that $f_{n_k}(x) \to f_\star(x)$ a.e.  If $f_\star(x)$ is continuous on an interval $[b,c] \subset [0, \pi]$, then the convergence is uniform on that interval.
\end{prop}

\begin{proof}

Suppose first that
$\mathcal{A}=\mathcal{S}_M$.  Let $(f_{n},a_{n})$ be a
sequence of functions in 
$\mathcal{S}_M$ with $a_n$ the corresponding values as in
Definition \ref{fnclasses}.
By the Bolzano-Weierstrass theorem, there exists a first subsequence $(a_{n_{l}})$ of $(a_n)$ such that 
$a_{n_{l}}\longrightarrow a_{\star}$.
There may be more than one point of accumulation $a_{\star}$ but by taking the least of them and passing to a further subsequence,
if necessary,
we can obtain a sequence $f_{n_{l}}$ that is non-increasing on $[0,a_{\star}]$ and
non-decreasing on $[a_{\star},\pi]$.
We can now invoke a theorem of Helly \cite{Doob},  \S X.9 by which, for any sequence of 
uniformly bounded monotonic functions on a 
fixed, finite interval, there exists
$f_{\star}\in \mathcal{S}_M$ such that for a further subsequence,
$f_{n_{l}}(x)\searrow f_{\star}(x) \quad $ for all $x$ 
(uniformly on any compact interval $I$ on which $f_{\star}$ is continuous).

Suppose now that $\mathcal{A}=\mathcal{C}_M$.  Then the  Blaschke Selection Theorem
implies that for every $f_{n}(x)\in \mathcal{C}_M$,
there exists a uniformly convergent subsequence $(f_{n_{l}})$ of $(f_n)$, such that
$$f_{n_{l}}\longrightarrow f_{\star}, \quad \mathsf{with }\  f_{\star}\in \mathcal{C}_M.$$
{(In the case of convex functions, continuity is automatic on any strict subinterval of $[0, \pi]$.)}
\end{proof}

\begin{thm}\label{exist}

Let $\mathcal{A}$ be any of the function classes of 
Definition {\rm\ref{fnclasses}}, and
suppose that $V_1 \in \mathcal{A}$.
Consider the eigenvalue problem \eqref{eP} where $V = V_0 + V_1$, where one part is fixed,
$V_0(x) \ge C > - \infty$, and
standard boundary conditions of the type \eqref{bc} are imposed at either end point
$0$ or $\pi$ if it is in the
limit-circle case.
(Due to the boundedness of $V_1$ this depends only on $V_0$.)
Then the
eigenvalues $\lambda_k$ of 
\eqref{eP} are continuous with respect to the topology of pointwise convergence for $V \in \mathcal{A}$,
and there is a potential $V_{\text{min},*} \in \mathcal{A}$ that minimizes $\Gamma(p,V)$
for $V = V_0 +V_1$ with $V_1 \in \mathcal{A}$.
\end{thm}

\begin{rems}
\item
1.  This can be regarded as a variant of Theorem II.1 of \cite{AsHaSv}, where
in addition we control for the single-well or convex assumption.  We caution that the
optimizers are not claimed to be unique.
\item
2.  We note that, likewise, a minimizing potential exists for
the eigenvalue gap for \eqref{eP} with $V = V_0 - V_1$.  Moreover, 
potentials $V_{\text{max},*,\pm} \in \mathcal{A}$ exist that maximize $\Gamma(p,V)$ for
$V = V_0 \pm V_1$.
\end{rems}

\begin{proof}

Let $(V_{n},a_{n})$ be a minimizing sequence for
the functional
\[
\Gamma (V)=\lambda_{2}(V)-\lambda_{1}(V).
\]
We know by 
Proposition \ref{subseq} that by passing to a subsequence, 
 $V_{n_{l}}(x) \to V_{\star}(x)$
 for all $x$, with $ V_{\star}(x) \in \mathcal{A}$.
 Because $\mathcal{A} \subset L^p[0, \pi]$ for each $p$, $1 \le p \le \infty$, 
 and $L^p$ convergence follows by the Lebesgue Dominated 
Convergence Theorem (for
 any given $p < \infty$),
 we can now follow the proof of of Theorem II.1 of \cite{AsHaSv} to
 establish continuity of $\Gamma$ and finish the claim.
\end{proof}

Once existence has been established,
one can characterize the optimizers in some of the listed cases by a variational analysis, using the Feynman-Hellman formula.
(See \eqref{FH}, below.)

\section{Characterization of optimizers}

In this section we characterize the optimizers of some of the problems in Proposition \ref{subseq} by introducing a set of 
perturbations $P(x)$ which would lead to a contradiciton to $\Gamma'(0) = 0$ unless $V_\star$ has special properties.

\subsection{The class of single-well potenials}
\begin{theorem}\label{StepFnMinimizesSM}

For any $M > 0$, the potential $V_1 \in \mathcal{S}_{M}$ that minimizes
$\Gamma(V_0+V_1)$, with $V_0$ as described in the Introduction, is
$M$ times the indicator function of a strict subinterval interval containing either $0$ or $\pi$.
The minimal gap $\Gamma_\star(M)$ is a decreasing function of $M$, and in the 
classical case where $p=1$,  $V_0 = 0$, and Dirichlet conditions are imposed at $0$ and $\pi$,
we have the following characterization of the gap minimizers.
\begin{enumerate}
\item
The potential energy functions that minimize the gap in the category $\mathcal{S}_{M}$ are of the form
$V_\star(x, M) := M \chi_{x_-(M)}(x)$ and 
$V_\star(\pi-x, M)$, for a function $x_-(M) < \frac \pi 2$ uniquely defined by an
explicit system of equations.
\item
For the gap-minimizing operators, 
\[
M+1 < \lambda_2 < M+4, \quad \text{and} \quad  M-2 < \lambda_1 < M+1,
\]
and for sufficiently large $M$, $\lambda_1 < M$.
(Again, the exact values of $\lambda_{1,2}$ are determined by explicit transcental equations.)
\item
If $M > \frac 7 2$, $\frac{\pi}{2 \sqrt{M}} \le x_-(M) \le \frac{\pi}{\sqrt{M-2}}$.
\item
$\lim_{M \to \infty}\Gamma_\star(M) = \left(\frac{\theta}{\pi}\right)^2 \doteq 2.04575$,
where $\theta$ is the first positive solution of $\theta = \tan \theta$.  There is no single-well potential such that this
infimum is attained.
\end{enumerate}
\end{theorem}

For $V \in \mathcal{S}_{M}$ let
$H(p,V)=-\frac{d}{dx}p(x)\frac{d}{dx}+V(x)$, for which
$$\sigma(H(p,V))=\left\lbrace \lambda_{1}< \lambda_{2}<\cdots \right\rbrace,$$
with corresponding normalized $u_{1},u_{2},\dots$, $u_{1}>0$ on $[0,\pi]$.
We may choose a sign for $u_2$ such that
for some $x_{0}$,
\[
\begin{array}{lr}
   u_{2}(x)>0 &\mbox{on}\quad (0,x_{0}),\\
   u_{2}(x)<0 &\mbox{on}\quad (x_{0},\pi).
\end{array}
\]

We will make heavy use of first-order perturbation theory to
characterize the effect on the eigenvalues of a small
{change in potential energy $V(x)$. Thus,} let $V\in \mathcal{S}_{M}$, and $V(x,\kappa)$ be a one-parameter $\frac{\partial V(x,\kappa)}{\partial \kappa}$ family of functions in $\mathcal{S}_{M}$ such
that $\frac{\partial V(x,\kappa)}{\partial \kappa}$ exists as a bounded, measurable function.
Let $\lambda_{n}(\kappa)$ denote
the n-th eigenvalue of the Schr\"odinger operator with potential $V(x,\kappa)$. If $\lambda_{n}(\kappa)$ is a simple
eigenvalue , then the standard ``Feynman-Hellman formula'' of perturbation theory states that
\begin{equation}\label{FH}
\frac{d\lambda_{n}(\kappa)}{d\kappa}=
\int_{0}^{\pi}\frac{\partial V(x,\kappa)}{\partial \kappa}u_{n}^{2}(x,\kappa)dx.
\end{equation}
(See, for instance, \cite{Kato}, \S II.2.)
We are guaranteed that $\lambda_{1}$ is simple, but this is not 
necessarily true of $\lambda_{2}$ and
hence also not automatic for
$\Gamma=\lambda_{2}-\lambda_{1} $. Nonetheless, according to degenerate perturbation theory, 
cf. \cite{Kato} \S VII.6, if $\lambda_{2}$   is $\mathit{l}$-fold degenerate 
at $\kappa=\kappa_0$, then there is a relabeling of the eigenvalues  $\lambda_{2}\ldots \lambda_{2+l-1}$ and of the basis of eigenvectors, so that $\hat{\lambda}_{2}\ldots \hat{\lambda}_{2+l-1}$ are analytic functions in a neighborhood of $t_0$, and
\eqref{FH} remains valid
with the specified basis of eigenvectors. (In the situation of this article, $l\leq 2$ and we shall always have
$\kappa_0=0$.) We note for future purposes that $\lambda_{2}=\min (\hat{\lambda}_{2},\hat{\lambda}_{3})$.

The following lemma is a
straightforward extension of a result in \cite{AsBe}, which was central to \cite{Lav}.

\begin{lemma}\label{lemma2.1}
The equation $u_{2}^{2}(x)-u_{1}^{2}(x)=0$ has at most two solutions
in $(0,\pi)$.
\end{lemma}

\begin{proof}
Let $x_{0}$ be the
unique zero  of $u_{2}=0$ in $(0,\pi)$.

Suppose that there
exist $\alpha_{1},\alpha_{2} \in (0,x)$ such that : $$\left\vert u_{2}(\alpha_{i})\right\vert =\left\vert u_1(\alpha_{i})\right\vert ;\quad i=1,2.$$
Define $v(x)=\frac{u_{2}(x)}{u_{1}(x)}$, so $v(\alpha_{1})=v(\alpha_{2})=1$.
By Rolle's Theorem, there exists $\xi \in (\alpha_{1},\alpha_{2})\subset (0,x_{0})$ such that 
\begin{equation}\label{Rolle}
v'(\xi)=0.
\end{equation}
Defining the Wronskian  $W(x)=u_1(x)p(x)u'_2(x)-u_2(x)p(x)u'_1(x)$,
we get
\[
W'(x)=( \lambda_{2}-\lambda_{1}) u_{1}(x)u_{2}(x).
\]
On the other hand, for all $x\in (0,x_{0})$,
\begin{align*}
\left( \frac{u_{2}(x)}{u_{1}(x)}\right)'&=-\frac{W(x)}{p(x)u_{1}^{2}(x)}\nonumber \\
&= \frac{1}{p(x) u_{1}^{2}(x)}\int_{0}^{x}\left( \lambda_{1}-\lambda_{2}\right)u_{1}(s)u_{2}(s)ds <0, \nonumber
\end{align*}
which contradicts \eqref{Rolle}.
Hence there is at most one
solution
of the equation $\left\vert u_{2}(x)\right\vert=\left\vert  u_{1}(x)\right\vert$ in $(0,x_{0})$,
which implies that there is at most one zero of the equation $$u_{2}^{2}(x)-u_{1}^{2}(x)=0$$
for $x\in (0, x_{0})$.  If we reflect the interval so that $x \leftrightarrow \pi-x$ and repeat the argument, we can
conclude that there is at most one zero of the equation $\left\vert u_{2}(x)\right\vert=\left\vert  u_{1}(x)\right\vert$ in $[x_{0},\pi)$.
Consequently, there are at most two zeroes of the equation
\[u_{2}^{2}(x)-u_{1}^{2}(x)=0
\]
in $(0, \pi)$.
\end{proof}

By Lemma \ref{lemma2.1}, there 
exist $x_{\pm}: 0\leq x_{-}< x_{0}< x_{+}\leq \pi$, for which
\begin{equation}\label{eq2.1}
 u_{2}^{2}(x)-u_{1}^{2}(x) =
     \begin{cases}
        >0 & x\in (0,x_{-})\cup (x_{+},\pi) \\
        <0  & x\in (x_{-},x_{+}).
     \end{cases}
\end{equation}

\noindent
\textbf{Proof of Theorem \ref{StepFnMinimizesSM}.}
It has been established by compactness that $\mathcal{S}_{M}$ contains a function 
$V_\star$ such that
$V = V_0 + V_\star$ minimizes $\Gamma(V)$.

Let $u_{1}$ and $u_{2}$ be the first and second normalized eigenfunctions of the problem (\ref{eP}), respectively.  By
Lemma \ref{lemma2.1}, there exists $0\leq x_{-}< x_{+}\leq \pi$ satisfying (\ref{eq2.1}) corresponding to
$V = V_0 + V_{\star}$.

We consider first
the case: $x_{-}<a<x_{+}$.
For suitable perturbations $P$ we define $$V_{\star \kappa}(x)=(1-\kappa)V_{\star}+\kappa P(x).$$

In the formula
$$\Gamma^{'}(k)=\displaystyle\int_{0}^{\pi}\left( u_{2}^{2}(x)-u_{1}^{2}(x)\right)\left( P(x)-V_{\star}\right)dx$$
we choose
\begin{equation}
P(x)=
\begin{cases}
   V_{\star}(x) & \text{on}\quad [x_{-},x_{+}]^{c}  \\
   \max\left( V_{\star}(x_{-}),V_{\star}(x_{+})\right) &\text{on}\quad [x_{-},x_{+}].
\end{cases}
\end{equation}
It can be seen that the functions $V_{\star \kappa}(x) \in \mathcal{S}_{M}$ for $0 \le \kappa \le 1$, and that 
$P(x) - V_{\star}(x) \ge 0$ and is supported in $[x_-, x_+]$ (if not identically $0$).

If $V_{\star}$ is not constant on $[x_{-},x_{+}]$, then this implies that $\Gamma^{'}(0) < 0$, which contradicts the minimality of 
$V_{\star}$.
We conclude that $V_{\star}=\textrm{cst}$ a.e.\ in $(x_{-},x_{+})$.

Next for $0\leq x\leq x_{-}$, we choose $P(x)=V_{\star}(x_{-})$,
and otherwise set $P(x)=V_{\star}(x)$, with which
\begin{align*}
\Gamma^{'}(k)&=\int_{0}^{\pi}\left( u_{2}^{2}(x)-u_{1}^{2}(x)\right)\left( P(x)-V_{\star}(x)\right)dx \nonumber \\
&= \int_{0}^{x_{-}}\left( u_{2}^{2}(x)-u_{1}^{2}(x)\right)\left( V_{\star}(x_{-})-V_{\star}(x)\right)dx \leq 0, \nonumber
\end{align*}
and $\Gamma^{'}(k)<0$ unless $V(x)$ is constant on $(x_{-},x_{+})$.

For $[x_{+},\pi]$ the same argument applies.

The second possibility is that $a<x_{-}$.  We let 
\begin{equation}
P(x)=
\begin{cases}
   V_{\star}(a) & \text{if}\quad 0\leq x<a  \\
   V_{\star}(x) &\text{otherwise},
\end{cases}
\end{equation}
so that
\begin{align*}
\Gamma^{'}(k)&=\int_{0}^{\pi}\left( u_{2}^{2}(x)-u_{1}^{2}(x)\right)\left( P(x)-V_{\star}(x)\right)dx \nonumber \\
&=\int_{0}^{a}\left( u_{2}^{2}(x)-u_{1}^{2}(x)\right)\left( V_{\star}(a)-V_{\star}(x)\right)dx \leq 0, \nonumber
\end{align*}
and $\Gamma^{'}(k)<0$ unless $V(x)$ is constant on $(0,x_{-})$.
Alternatively, let
\begin{equation}
P(x)=
\begin{cases}
   V_{\star}(x) & \text{if}\quad  x<x_{+}  \\
   V_{\star}(x_{+}) &\text{otherwise},
\end{cases}
\end{equation}
with which
\begin{align*}
\Gamma^{'}(k)&=\int_{0}^{\pi}\left( u_{2}^{2}(x)-u_{1}^{2}(x)\right)\left( P(x)-V_{\star}(x)\right)dx \nonumber \\
&=\int_{x_{+}}^{\pi}\left( u_{2}^{2}(x)-u_{1}^{2}(x)\right)\left( V_{\star}(x_{+})-V_{\star}(x)\right)dx \leq 0. \nonumber
\end{align*}
Thus $\Gamma'(k)<0$ unless $V(x)=cst$ on $(x_{+}, \pi)$.

The statement that the minimal gap decreases monotonically with respect to $M$ is elementary, since the set over which the minimum is sought is larger for larger $M$.

We turn now to the characterization of the minimizers
under the assumptions that $p = 1$, $V_0 = 0$ and Dirichlet boundary conditions are imposed at $0$ and $\pi$.
The first item is what has been proved above, where we assumed without loss of generality, that the step function
equals $0$ on 
the interval $[0, x_-)$.  (The other possibility is covered by $x \leftrightarrow \pi - x$,  $x_-  \leftrightarrow \pi-x_+ $.)

In the next claim the upper bounds on $\lambda_{1,2}$ are immediate from the fact that the first and second eigenvalues with $V$ set to $0$ are $1$ and $4$,
and that 
$V_\star \le M$ a.e.  To establish that $\lambda_2 > M+1$, recall that, according to the argument 
given above, the support of the minimizing step function for $V \in \mathcal{S}_{M}$ contains
$x_0$, the unique interior zero of $u_2$.
Now, on the support of $V_\star$ the eigenfunction $u_2$ is a multiple of $\sin(\sqrt{\lambda_2 - M} (\pi - x))$, which cannot have a zero in $\textrm{supp}(V_\star)$
unless $\lambda_2 - M > 1$.
The estimate $\lambda_1 > M-2$ is a consequence of the fact that that the minimal value of
$\Gamma$ is less than the gap for the case $V\equiv0$, which equals 3.

For the upper bound on $x_-$, we use the Rayleigh-Ritz inequality to estimate $\lambda_1$, choosing for the test function
a normalization constant times $\chi_{[0, x_-]} \sin\left(\frac{\pi x}{x_-}\right)$.  A calculation of the Rayleigh quotient for this function yields that
\[
\lambda_1 \le \left(\frac{\pi}{x_-}\right)^2.
\]
When combined with the inequality $\lambda_1 > M-2$, the claimed upper bound on $x_-$ follows.  (It requires only 
$M > 2$ rather than $M > \frac 7 2$, to be used later.)

The lower bound on $x_-$ and the final estimate require a fine analysis of the transcendental equations solved by the 
eigenvalues.
On the interval $[0, x_-)$, the eigenfunctions that satisfy Dirichlet conditions are multiples of $\sin(\sqrt{\lambda}x)$ whereas 
on $(x_-, \pi]$ they are multiples of $\sin(\sqrt{\lambda-M}(\pi -x))$ (assuming for now that $\lambda > M$).  Since the eigenfunctions must be $C^1$, by equating their logarithmic derivatives 
at $x_-$, the eigenvalues are determined by the transcendental equation
\begin{equation}\label{transeq}
\sqrt{\lambda} \cot\left(\sqrt{\lambda} x_-\right) = - \sqrt{\lambda - M} \cot\left(\sqrt{\lambda - M} (\pi -x_-)\right).
\end{equation}
(This holds for any step-function potential, regardless of whether $x_-$ satisfies the condition to minimize the gap, $
\left(u_2(x_-)\right)^2 - \left(u_1(x_-)\right)^2 = 0$.)
If $\lambda < M$, then the same argument leads to
\[
\sqrt{\lambda} \cot\left(\sqrt{\lambda} x_-\right) = - \sqrt{M - \lambda} \coth\left(\sqrt{M - \lambda} (\pi -x_-)\right).
\]

There remains the possibility that $\lambda_1 = M$, and indeed an eigenfunction with this eigenvalue is possible if
the eigenfunction is 
a multiple of $\sin(\sqrt{M} x$) for $x \le x_-$ and linear on $[x_-, \pi]$.  The condition for this function to be $C^1$ is
that $x_-$ have the value for which $\sqrt{M} \cot(\sqrt{M} x_-) = - \frac{\cos(\sqrt{M} x_-)}{\pi - x_-}$, or, after simplification,
 $\sqrt{M} (\pi - x_-) = - \sin(\sqrt{M} x_-)$.  For large $M$, however, an eigenfunction of this type
 is in contradiction with the upper bound on 
$x_-$.  (A calculation using Mathematica shows that $M > \frac 7 2$
more than suffices to eliminate the possibility of eigenfunctions of this type.)

Since the transcendental equations
for $\lambda$ contain a large parameter $M$, we find it convenient to
rescale them.
With the substitutions
\begin{align*}
r^2 &:= \lambda - M\\
y &:= \sqrt{M} x_-\\
\mu &:= M^{-\frac 1 2}
\end{align*}
and some simple algebra,
\eqref{transeq} takes on the form
\begin{equation}\label{resctranseq}
\tan\left(r (\pi - \mu y)\right) + 
\mu \left(r/\sqrt{1 + \mu^2 r^2}\right) \tan\left(\sqrt{1 + \mu^2 r^2} y \right) = 0,
\end{equation}
provided that $\lambda > M \Leftrightarrow r > 0$.
In the case that $\lambda_1 < M$, a similar calculation with $s^2 := M - \lambda$ leads to
\begin{equation}\label{resctranseqLESS}
\tanh\left(s (\pi - \mu y)\right) + 
\mu \left(s/\sqrt{1 - \mu^2 s^2}\right) \tan\left(\sqrt{1 - \mu^2 s^2} y \right) = 0.
\end{equation}

 We continue now with a narrow examination of the pair of transcendental equations \eqref{resctranseq} and
 \eqref{resctranseqLESS} in the limit of small $\mu$.  For fixed $\mu$ and $y$, $\tan\left(r (\pi - \mu y)\right)$ increases monotonically between a succession of vertical asymptotes,
while the second term in \eqref{resctranseq} is continuous and of small magnitude.  It is easy to see
by considering the vertical asymptotes of the first function in \eqref{resctranseq}
that solutions for $r$
occur near the positive integers.  To have an eigenvalue gap $< 3$ we will need a
solution for $r$ in the interval containing $r=1$ to correspond to a sign-changing eigenfunction, which means that 
the ground-state eigenvalue, which does not change sign,
must correspond to a solution of  \eqref{resctranseqLESS} with $s > 0$, i.e., $\lambda_1 < M$.
A necessary condition for this, needed to make the second term negative, is for $y > \frac \pi 2$, which is equivalent to the
claimed lower bound on $x_-$ in the theorem.  At the same time, we must have $y \le \frac \pi 2  + O(\mu)$ to have a solution, because of the factor $\mu$ in the second term.
To obtain a precise asymptotic estimate, we therefore set $y = \frac \pi 2  + y_1 \mu + O(\mu^2)$
and expand \eqref{resctranseq} and \eqref{resctranseqLESS} with Taylor's formula, keeping only leading terms.  The results are
\begin{align*}
\tan(\pi r)\left(1 + O(\mu)\right)  &= - \frac{\mu r}{\cot\left(\frac \pi 2 + y_1 \mu\right)}\left(1 + O(\mu^2)\right)\\
&= \frac{\mu r}{\tan(y_1 \mu)}\left(1 + O(\mu^2)\right)\\
&= \frac{r}{y_1}\left(1 + O(\mu^2\right),
\end{align*}
and, similarly,
\begin{equation}\nonumber
\tanh(\pi s)\left(1 + O(\mu)\right)  = \frac{s}{y_1}\left(1 + O(\mu^2\right).
\end{equation}

Since our task is to minimize $\Gamma = r^2 + s^2$ as a function of $y$
as $\mu \to 0$, we may neglect
higher-order terms and instead minimize $r^2 + s^2$, defined by the solutions of
\begin{equation}\label{resceq-r}
\tan(\pi r)  = \frac{r}{y_1}
\end{equation}
and
\begin{equation}\label{resceq-s}
\tanh(\pi s)  = \frac{s}{y_1}.
\end{equation}
Eq.~\eqref{resceq-s} has no solutions for $s \ge 0$ unless $y_1 \ge \frac 1 \pi \doteq 0.31831$.  A simple
upper limit on 
$y_1$ is provided by the value corresponding to $s = \frac 3 2$ in  \eqref{resceq-s}, {\em viz}., 
\[
\frac{3}{2 \tanh\left(\frac{3 \pi}{2}\right)} \doteq 1.50024.
\]
By differentiating \eqref{resceq-r} and \eqref{resceq-s}, applying the identities that 
\[
\frac{d \tan(z)}{dz} = 1 + \tan(z)^2, \quad\quad \frac{d \tanh(z)}{dz} = 1 - \tanh(z)^2,
\]
and algebraically arranging terms, we obtain
\[
\frac{d \left(r^2 + s^2\right)}{d y_1} = \frac 2 \pi \left(\frac{s^2}{s^2 - \eta}-\frac{r^2}{r^2 + \eta}\right)
\]
with $\eta := y_1 \left(y_1 - \frac 1 \pi\right) \ge 0$.
Hence a critical point of the function $r^2 + s^2$ must satisfy
\[
\frac{s^2}{s^2 - \eta}-\frac{r^2}{r^2 + \eta} = 0,
\]
which is equivalent to
\begin{equation}\label{cpcond}
\frac{s^2}{r^2} = \frac{s^2 - \eta}{r^2 + \eta}.
\end{equation}
It is evident that a critical point requires $\eta=0  \Leftrightarrow y_1 = \frac 1 \pi$ (which furthermore implies $s=0$).  Since 
a calculation shows that $r(1.50024)^2 + s(1.50024)^2   \doteq 4.8171$, which is larger than 3, and there are no critical points
for smaller values of $y_1 > \frac 1 \pi$, the remaining possibility for the minimal gap is $ y_1 = \frac 1 \pi$.  
Having determined the next-order correction to $y$, we conclude:
\[
V_\star(M) = M \chi_{x_-(M)}(x)
\]
with
\[
x_- = \frac{\pi}{
2 \sqrt{M}} + \frac{1}{\pi M} + \cdots,
\]
and
\[
\Gamma(V_\star(M)) = \left(\frac{\theta}{\pi}\right)^2,
\]
proving the lower bound claimed in the theorem.  We finally note that as $\mu \to 0$ the gap-minimizing 
potential energy does not have a sensible limit.

\subsection{The class of convex potentials}
We show in this section how Lavine's result can be obtained
and extended by our methods.

\begin{prop}
Consider $V = V_0 + V_1$ where $V_1 \in \mathcal{C}$ or $\mathcal{C}_M$.
If $V_1 = V_\star$ minimizes $\Gamma(V)$, then
$V_{\star}$ cannot be strictly convex
on any interval.
\end{prop}

\begin{proof}
Suppose that $V_{\star}$ were
convex on an interval $I$.  By passing to a subinterval if necessary, we can
arrange that
$u_{2}^{2}-u_{1}^{2}$ does not change sign on $I$.  We can then choose $P(x)\in C^{2}$, $\textrm{supp}\, P(x)\subset I$ and $P(x)(u_{2}^{2}-u_{1}^{2})<0$
on $I$.  By  \eqref{FH} $\Gamma^{'}<0$, which is a contradiction.
\end{proof}

\begin{prop}
$V_{\star}=mx+b$.
\end{prop}

\begin{proof}
We begin by recalling that a convex function has right and left derivatives at every interior point of an interval, with only at most 
a countable number of points at which the right and left derivatives can differ.  Restating this in the language of distributions,
we may write
\[
V_{\star}''=\sum\alpha_{n}\delta(x-x_{n}),\quad \alpha_{n}\geq 0.
\]
We need to show that all $\alpha_{n}=0$.

Suppose that there exists $ x_{n}\leq x_{-}$ or $x_n \geq x_{+}$ with $\alpha_{n}>0$. Then, assuming $x_{n}\leq x_{-}$,
consider a perturbation $P(x):=(x-x_{n})\chi)_{[0,x_{n}]}(x)$, which preserves convexity.  Then
\[
\Gamma'(0)=\int_{0}^{x_{n}}(x-x_{n})(u_{2}^{2}-u_{1}^{2}) dx <0,
\]
contradicting the assumed minimality of $\Gamma(V_0 + V_\star)$.
The same argument works assuming $x_{n}\geq x_{+}$.

If $x_{-}< x_{n}<x_{+}$, then we take 
\begin{equation}
P(x)=
\begin{cases}
   \frac{x-x_{-}}{x_{n}-x_{-}} & \text{if}\quad  x<x_{n}  \\
   \frac{x_{+}-x}{x_{+}-x_{n}} & \text{if}\quad x\geq x_{n}.
\end{cases}
\end{equation}
For small $t$, $V_{\star}(x)+tP(x)$ is convex, but again 
$$\Gamma'(0)=\int_{0}^{x_{n}}(x-x_{n})(u_{2}^{2}-u_{1}^{2}) dx <0,$$
so this possibility also contradicts the assumed minimality.
\end{proof}

When $V_0 = 0$ as in \cite{Lav}, a particular argument is needed to establish that $m=0$, i.e., that the optimal
potential is a constant.  This is generally not the case when there is a background potential $V_0$, as can be seen from the 
trivial example where $V_0 = x$:  In this example, $V_\star$ must equal $- x + b$, because of the original result of \cite{Lav}.

Lastly, we turn to the question about what happens if there is no finite upper bound $M$.  In prior work
going back to Lavine it has been customary to assume
without comment that $V(x)$ is continuous on the closed interval 
$[0, \pi]$, and therefore bounded.  This leaves open the possibility that there are convex, resp. single-well potentials, which diverge to $+\infty$ as $x$ tends to $0$ or $\pi$, which have lower fundamental gaps than 
the lim-inf of the minimal gaps for $M < \infty.$  In our next result we show that this is not possible.  For simplicity we restrict ourselves to Liouville normal form and Dirichlet boundary conditions.

\begin{theorem}
Consider the eigenvalue problem
\[
H(1, V) u = \lambda u,
\]
with homogeneous Dirichlet boundary conditions at $0$ and $\pi$, and $V(x) \ge 0$
either convex or of single-well form.  Then for any $\epsilon > 0$ and $M > \lambda_2(V)$, 
there exists $V_{M, \epsilon} \in \mathcal{S}_M$ or respectively 
$V_{M, \epsilon} \in \mathcal{C}_M$, such that
$\lambda_k(V) - \epsilon \le \lambda_k(V_M) \le \lambda_k(V)$
for $k=1,2$. 
\end{theorem}

\begin{proof}
We consider the case where $\lim_{x \downarrow 0}V(x) = +\infty$; the case where there is a divergence
as $x \uparrow \pi$ is the same after a simple change of variable.  We also 
assume without loss of generality that $\min V = 0$ and that $u_{1,2} > 0$ on a neighborhood of $0$.
We begin by observing if $V \ge 0$ on $[0, \pi]$, then there exists a constant $C$
{\em independent of $V$} such that the normalized eigenfunctions satisfy
\[
\|u_k\|_{\infty} \le C \lambda_k^{\frac 1 4}.
\]
(In Example 2.1.8 of \cite{Dav} this is shown with $C =  \mathrm{e}^{1/8 \pi}$.)

Because $V$ diverges at $0$, there is an interval $[0, \ell]$, $\ell > 0$ on which $V - \lambda_2 \ge 0$, and therefore $u_{1,2}$ are convex.  In particular, for any 
$\ell_0 < \ell$, 
if $x < \ell_0$, then
\[
\psi'(x) \le \frac{\psi(\ell)}{\ell-x}.
\]
This estimate holds for
{\em any} $\tilde{V}$ such that
$\tilde{V} > \lambda_2$ for $x < \ell_0 < \ell$, 
not only for the originally assumed $V$.

We first consider the case of single-well potentials and finish the argument.
Fix some $M > \lambda_2(V)$ and consider the perturbation where 
\[
V(x) \to V(\kappa, x) := (1-\kappa) V(x) + \kappa M
\]
when $x \le \ell_1$ but leaving $V$ unchanged for larger $x$.  
This family of functions is of single-well form for all 
$0 \le \kappa \le 1$ and is nonincreasing in $\kappa$.
Here $\ell_1 < \ell/2$ will be chosen later depending on 
$\epsilon$.  According to \eqref{FH},
\[
\lambda_k^\prime(\kappa) = - \int_0^{\ell_1} u_k^2(x) \left(V(x) - M\right) dx < 0,
\]
for all $0 \le \kappa \le 1$.  On the other hand,
\begin{align*}
\lambda_k^\prime(\kappa) &= - \int_0^{\ell_1} u_k^2(x) \left(V(x) - M\right) dx \\
&\ge - \int_0^{\ell_1} u_k^2(x) \left(V(x) - \lambda_k(\kappa) \right) dx\\
&= - \int_0^{\ell_1} u_k(x) u_k^{\prime\prime}(x) dx\\
&= - \int_0^{\ell_1} \left(u_k(x) u_k'(x)\right)^\prime dx
+ \int_0^{\ell_1} \left(u_k^\prime(x)\right)^2 dx\\
&\ge -u_k(\ell_1)u'_k(\ell_1).
\end{align*}
(For notational simplicity we have not indicated explicitly the dependence of these
eigenfunctions on $\kappa$.)
Appealing to the convexity of the eigenfunctions and knowing that their maximum is not attained 
in the interval $[0, \ell]$, we get
\[
\lambda_k^\prime \ge - \left(C \lambda_2^{\frac 1 4}\right)^2 \frac{\ell_1}{\ell} \frac{\ell/2}{\ell/2 - \ell_1}.
\]
Finally, for any $\epsilon > 0$, $\epsilon < 1$, we can choose 
\[
\ell_1 \le \frac{\ell}{\left(C \lambda_2^{\frac 1 4}\right)^2 + 2} \,\,\epsilon
\]
to conclude that 
\[
\lambda_k^\prime \ge - \epsilon.
\]
The claim follows by defining $V_M = V(1,x)$ and observing that
\[
\lambda_k(V_m) = \lambda_k(V) + \int_0^1 \lambda^\prime(\kappa) \mathrm{d}\kappa. 
\]

A closely 
similar argument yields the same conclusion for convex potentials, except that the perturbation must
now preserve convexity rather than the single-well property.  A suitable family of perturbations is
\[
V(x) \to V(\kappa, x) := (1-\kappa) V(x) + \kappa \left(V(\ell) + V^\prime(\ell)(x-\ell)\right)
\]
on a sufficiently small interval $[0, \ell]$ to guarantee that 
$V_{M, \epsilon} > M > \lambda_2(M)$.  We omit the details, being quite similar to the case treated above.
\end{proof}

In particular, when $V_0 = 0$, then Lavine's result that the constant potentials are the (only) minimizers of the fundamental gap holds
without the assumption that $V(x)$ is continuous at $x = 0$ and $\pi$, at least with Dirichlet boundary conditions.

\medskip
{\bf Acknowledgments.}  The authors wish to thank 
Joachim Stubbe and Timo Weidl for useful conversations and insights.


\begin{thebibliography}{00}

\bibitem{AsBe}
\newblock {M. Ashbaugh and R. Benguria, 
\em Optimal lower bound for the gap between the first two eigenvalues of one-dimensional  Schr\"odinger operators with symmetric single-well potentials, }
\newblock Proc. Amer. Math. Soc. 105 (1989), 419--424.

\bibitem{AsHa82}
\newblock {M.~S. Ashbaugh and E.~M. Harrell II,
\em Perturbation theory for shape resonances and large
barrier potentials},
\newblock Commun. Math. Phys. 83 (1982) 151--170.

\bibitem{AsHaSv}
\newblock {M.~S. Ashbaugh, E.~M. Harrell II, and R. Svirsky,
\em On minimal and maximal eigenvalues gaps and their causes, }
\newblock Pac. J. Math. 147 (1991) 1--24.

\bibitem{BiRo}
\newblock {G. Birkhoff and G.-C. Rota}
\newblock Ordinary differential equations, 4th Ed.,
John Wiley \& Sons, New York, 1989.

\bibitem{BrEaSc}
\newblock {B.~M. Brown,  M.~S.~P. Eastham, and  K.~M. Schmidt,
\em Periodic Differential Operators}
\newblock Operator Theory Advances and Applications (230) - 2013, Birkh\"auser.

\bibitem{Dav}
E.~B. Davies,
{Heat Kernels and Spectral Theory},
Cambridge University Press, Cambridge, 1989.
 
\bibitem{Doob}
\newblock {J. L. Doob,
\em Measure theory},
volume~143 of {\em
  Graduate Texts in Mathematics}.
\newblock Springer-Verlag, New York-Berlin, 1994.
 
 \bibitem{Kato}
\newblock {T. Kato,
\em Perturbation theory for linear operators,}
\newblock  New York: Springer-Verlag, 1980.

 \bibitem{Lav}
 \newblock {R. Lavine,
 \em The eigenvalue gap for one-dimensional convex potentials,}
 \newblock Proc. Amer. Math. Soc. 121 (1994) 815--821. 
 
\bibitem{Horvath}
\newblock {M. Horv\'ath
\em On the first two eigenvalues of Sturm-Liouville operators,}
\newblock Proc. Amer. Math. Soc. 131 4 (2002) 1215--1224.
 
\bibitem{Tes}
G. Teschl,
\newblock {\em Mathematical methods in quantum mechanics}, volume~99 of {\em
  Graduate Studies in Mathematics}.
\newblock Amer. Math. Soc., Providence, 2009.

\bibitem{Wei}
J. Weidmann,
\newblock {\em Linear operators in {H}ilbert spaces}, volume~68 of {\em
  Graduate Texts in Mathematics}.
\newblock Springer-Verlag, New York-Berlin, 1980.
\newblock Translated from the German by Joseph Sz{\"u}cs.

 \bibitem{YuYa}
\newblock {X. J. Yu and C. F. Yang,
\em The gap between the first two eigenvalues of Schr\"odinger operators with single-well potential,}
\newblock Appl. Math. Comp. 268 (2015) 2750--283.
 
\bibitem{Zet}
\newblock A. Zettl
\newblock {\em Sturm-Liouville Theory}, volume 121 of
{\em Mathematical Surveys and Monographs}.
\newblock Amer. Math. Soc., Providence, 2005.
 
 
\end{thebibliography}
\end{document}